\documentclass[11pt]{nyjm}
\usepackage{amsmath}
\usepackage{amssymb}
\usepackage{amsthm}
\usepackage{latexsym,tikz}
\usepackage{hyperref}
\usepackage{enumerate}
\usepackage[all]{xy}

\newcommand{\diag}{{diag}}
\newcommand{\R}{\mathbb{R}}
\newcommand{\C}{\mathbb{C}}
\newcommand{\F}{\mathbb{F}}
\newcommand{\D}{\mathbb{D}}
\newcommand{\bH}{\mathbb{H}}
\newcommand{\Hq}{\mathbb{H}}
\newcommand{\bP}{\mathbb{P}}
\newcommand{\bS}{\mathbb{S}}
\newcommand{\M}{\mathcal{M}}
\newcommand{\Conv}{{\rm conv}}
\newcommand{\iConv}{{\rm iconv}}
\newcommand{\bD}{\mathbb{D}}

\newcommand{\vc}[1]{\boldsymbol{#1}}

 \newtheorem{thm}{Theorem}[section]
 \newtheorem{cor}[thm]{Corollary}
 
 \newtheorem{lemma}[thm]{Lemma}
 \newtheorem{prop}[thm]{Proposition}
 \theoremstyle{definition}
 \newtheorem{defn}[thm]{Definition}
 \theoremstyle{remark}
 
 \newtheorem{ex}[thm]{Example}
 \numberwithin{equation}{section}

\begin{document}

\thanks{The second author was partially supported by FCT
 through project UID/MAT/04459/2013 and the third author was partially supported by FCT through CMA-UBI, project PEst-OE/MAT/UI0212/2013.}

\title[Convexity and circularity of the numerical range]{On the convexity and circularity of the numerical range of nilpotent quaternionic matrices}

%----------Author 1
\author[L. Carvalho]{Lu\'{\i}s Carvalho}
\address{Lu\'{\i}s Carvalho, ISCTE - Lisbon University Institute\\    Av. das For\c{c}as Armadas\\     1649-026, Lisbon\\   Portugal}
\email{luis.carvalho@iscte-iul.pt}
\author[Cristina Diogo]{Cristina Diogo}
\address{Cristina Diogo, ISCTE - Lisbon University Institute\\    Av. das For\c{c}as Armadas\\     1649-026, Lisbon\\   Portugal\\ and \\ Center for Mathematical Analysis, Geometry,
and Dynamical Systems\\ Mathematics Department,\\
Instituto Superior T\'ecnico, Universidade de Lisboa\\  Av. Rovisco Pais, 1049-001 Lisboa,  Portugal
}
\email{cristina.diogo@iscte-iul.pt}
\author[S. Mendes]{S\'{e}rgio Mendes}
\address{S\'{e}rgio Mendes, ISCTE - Lisbon University Institute\\    Av. das For\c{c}as Armadas\\     1649-026, Lisbon\\   Portugal\\ and Centro de Matem\'{a}tica e Aplica\c{c}\~{o}es \\ Universidade da Beira Interior \\ Rua Marqu\^{e}s d'\'{A}vila e Bolama \\ 6201-001, Covilh\~{a}}
\email{sergio.mendes@iscte-iul.pt}
%----------classification, keywords, date

\date{\today}

\maketitle

\begin{abstract}
We provide a sufficient condition for the numerical range of a nilpotent matrix $N$ to be circular in terms of the existence of cycles in an undirected graph associated with $N$. We prove that if we add to this matrix $N$ a diagonal real matrix $D$, the matrix $D+N$ has convex numerical range. For $3 \times 3$ nilpotent matrices, we strength further our results and obtain necessary and sufficient conditions for circularity and convexity of the numerical range.
\end{abstract}

%\tableofcontents

\maketitle
%%% ----------------------------------------------------------------------
%\tableofcontents
\section{Introduction}

Let $\bH$ denote the Hamilton quaternions and let $A$ be a $n\times n$ matrix with quaternionic entries. The quaternionic numerical range of $A$, denoted $W(A)$, was introduced in 1951 in Kippenhahn's seminal article \cite{Ki} as an analogue of the long established complex numerical range (see \cite{R} for an account on quaternionic numerical ranges). Specifically, $W(A)$ is the subset of $\bH$ whose elements have the form $\vc{x}^*A\vc{x}$, where $\vc{x}$ runs over the unit sphere of $\bH^n$. Due to the failure of Toeplitz-Hausdorff theorem in the quaternionic setting, the convexity of $W(A)$ has been studied by several authors. In \cite{Ki}, Kippenhahn introduced the Bild of $A$, denoted $B(A)$, as the intersection of $W(A)$ with the complex plane and studied its convexity. The Bild is indeed a planar substitute of $W(A)$ in the sense that every element of $W(A)$ is equivalent to an element in $W(A)\cap\C$.

The first remarkable result on convexity is due to Au-Yeung who proved in 1984 that $W(A)$ is convex if, and only if, the projection of $W(A)$ over $\R$ (resp., $\C$) equals the real (resp., the complex) elements in $W(A)$, see \cite[theorems 2 and 3]{Ye1}. In that same paper, the author gives necessary and sufficient conditions on the eigenvalues of a normal matrix $A$ for $W(A)$ to be convex. The convexity of the Bild, already an issue in \cite{Ki}, was established for normal matrices in 1994 by So, Thompson and Zhang. They proved in \cite[p. 192]{STZ} that the closed upper half plane part of the Bild (the upper Bild $B^+(A))$ of a normal matrix $A$ is the convex hull of eigenvalues and cone vertices. Later on, the proof was simplified by Au-Yeung \cite{Ye2}.

The general case was settled by So and Thompson in 1996. In \cite[theorem 15.2]{ST} they proved that for any matrix $A$, the intersection of $W(A)$ with the closed upper half plane is always convex.

Another problem that attracted much attention in the complex setting is the shape of the numerical range. In \cite[theorem 17.1]{ST}, So and Thompson characterized the numerical range of $2\times 2$ quaternionic matrices, the analogue of the elliptical range theorem. However, compared with the complex case, the shape of the numerical range of quaternionic matrices seems to have been more neglected.

In this article we study the convexity and shape of the numerical range for nilpotent quaternionic matrices. To be more specific, we determine under what conditions $W(A)$ has circular shape or, at least, is convex. In section 2 of this article we recall some definitions and fix notation. In section 3 we deal with the circularity of the numerical range. Theorem \ref{center at origin} shows that if the numerical range of a nilpotent matrix is a disk, its center must be located at the origin. This is the quaternionic analogue of \cite[proposition 1]{MM}. We conclude this section with theorem \ref{tree circular disk} which says that a sufficient condition for the numerical range of a nilpotent matrix $A$ to be a disk is that the associated graph of $A$ is a tree. This condition is not necessary as example \ref{ex_4x4realmatrix} shows. In section 4 we extend the results of the previous section. Every matrix  $A\in\mathcal{M}_n(\bH)$ is, up to unitary equivalence, upper triangular and every upper triangular matrix decomposes as a sum of a diagonal with a nilpotent matrix. The main result of this section is theorem \ref{W_is_smaller}, where it is proved that when the diagonal part is real and the nilpotent part is a tree then the numerical range is a union of disks. To reach this result we apply Berge's maximum theorem, a technique not much seen in the literature. Corollary \ref{prop_diag+shiftlike_convex} proves that this class of matrices have convex numerical range. We end the section providing an example of one of these matrices where the union of disks that compose its numerical range is in fact an ellipse. In section 5, we focus on $3\times 3$ nilpotent matrices concerning convexity and circularity of the numerical range. Theorem \ref{NR_3x3_disk} says that a necessary and sufficient condition for the numerical range of $A\in\mathcal{M}_3(\bH)$ to be a disk with center at the origin is that $A$ is cycle-free. On the other hand, theorem \ref{NSC 3x3 convex} gives a necessary and sufficient condition for the same class of matrices to have convex numerical range. Specifically, $W(A)$ is convex if, and only if, $a_{13}^*a_{12}a_{23}\in\R$. The link with our work is now provided by theorem 1 of Chien and Tam \cite{CT} in the complex setting.

\section{Preliminaries and notation} \label{section_prelims}

In this section we present some well known facts about quaternions and fix some notation. The quaternionic skew-field $\bH$ is an algebra of rank $4$ over $\R$ with basis $\{1, i, j, k\}$. The product in $\bH$ is given by $i^2=j^2=k^2=ijk=-1$. Denote the pure quaternions by $\bP=\mathrm{span}_{\R}\,\{i,j,k\}$. For any $q=a_0+a_1i+a_2j+a_3k\in\bH$ let $\pi_{\R}(q)=a_0$ and $\pi_{\bP}(q)=a_1i+a_2j+a_3k$ be the real and imaginary parts of $q$, respectively. The conjugate of $q$ is given by $q^*=\pi_{\R}(q)-\pi_{\bP}(q)$ and the norm is defined by $|q|^2=qq^*$. Two quaternions $q_1,q_2\in\bH$ are called similar if there exists a unitary quaternion $s$ such that $s^{*}q_2 s=q_1$. Similarity is an equivalence relation and we denote by $[q]$ the equivalence class containing $q$. A necessary and sufficient condition for the similarity of $q_1$ and $q_2$ is that $\pi_{\R}(q_1)=\pi_{\R}(q_2) \textrm{ and }|\pi_{\bP}(q_1)|=|\pi_{\bP}(q_2)|$  \cite[theorem 2.2.6]{R}.

Let $\F$ denote $\R$, $\C$ or $\bH$. Let $\F^n$ be the $n$-dimensional $\F$-space. For $\vc{x}\in\F^n$, $\vc{x}^*$ denote the conjugate transpose of $\vc{x}$.
The disk with center $\vc{a}\in\F^n$ and radius $r\geq 0$ is the set $\D_{\F^n}(\vc{a},r)=\{\vc{x}\in\F^n:|\vc{x}-\vc{a}|\leq r\}$ and its boundary is the sphere $\bS_{\F^n}(\vc{a},r)$.  In particular, if $\vc{a}=\vc{0}$ and $r=1$, we simply write $\D_{\F^n}$ and $\bS_{\F^n}$. The group of unitary quaternions is denoted by $\bS_{\bH}$. Notice that we are considering the singleton $\{\vc{a}\}$ to be the disk with center $a$ and radius $r=0$. Any $x \in \Hq$ can be written as $x=\beta z$, with $\beta=|x|$ and  $z \in \bS_{\Hq}$\footnote{If $x\neq 0$, $z$ is uniquely defined. However, if $x = 0$ that is not the case and we take $z=1$.}. We introduce the following notation:
\begin{align*}
&\R^{n,+} =\{(\beta_1,...,\beta_n)\in\R^n:\beta_i\geq 0, 1\leq i\leq n\}\\
&\bS^+_{\R^n} =\bS_{\R^n}\cap\R^{n,+}.
\end{align*}
Let $\M_{n} (\F)$ be the set of all $n\times n$ matrices with entries over $\F$.
Let $A\in \M_{n} (\Hq)$.
The set \[W(A)=\{\vc{x}^*A\vc{x}:\vc{x}\in \bS_{\Hq^n}\}\]
is called the numerical range of $A$ in $\Hq$. As usual,  the complex numerical range of a complex matrix is defined by \[W_{\C}(A)=\{\vc{x}^*A\vc{x}:\vc{x}\in \bS_{\C^n}\}.\]
%The numerical range is invariant under unitary equivalence, i.e., $W(A)=W(U^*AU)$, for every unitary $U\in \M_n(\Hq)$ (see \cite{R,Zh}).

It is well known that if $q\in W(A)$ then $[q]\subseteq W(A)$ \cite[page 38]{R}. Therefore, it is enough to study the subset of complex elements in each similarity class. This set is known as $B(A)$, the Bild of $A$
\[
B(A)=W(A)\cap\C.
\]
Although the Bild may not be convex, the upper Bild $B^+(A)=W(A)\cap\C^+$ is always convex, see \cite{ST}.

Taking into account that $\R$ can be seen as a real subspace of $\Hq$, what we denoted by $\pi_{\R}$ is, in fact, the projection of $\Hq$ over $\R$,  $\pi_{\R}:\Hq \rightarrow \R$. The projection of $W(A)$ over $\R$ is
\begin{equation*}
  \pi_{\R}(W(A))= \{\pi_{\R}(w): \; w\in W(A)\}.
\end{equation*}

In the next section we will define a relation between the circularity of the numerical range of $A$ and the lack of cycles of an associated undirected graph. To be more specific, given a matrix $A=[a_{ij}] \in \M_n (\bH)$ we may define the underlying undirected graph $\mathcal{G}_A$ with $n$ vertices as the graph with an edge between $i$ and $j$ whenever $a_{ij}\neq 0$ or $a_{ji}\neq 0$. That is, if $\delta:\Hq \to \{0,1\}$ is the indicator function, $\delta(q)=1$ if $q\neq 0$ and $\delta(q)=0$ otherwise, let $A_{\delta}$ be the symmetric matrix given by
\[
A_\delta=\Big[ \delta(\max\{a_{ij},a_{ji}\}) \Big]_{i,j=1}^n.
\]
Then $A_{\delta}$ is precisely the adjacency matrix of the undirected graph $\mathcal{G}_A$. We say that the graph $\mathcal{G}_A$ has a path between the vertices $i, j \in \{1, \ldots, n\}$, if there is a sequence of vertices $(i_1, i_2, \ldots, i_p)$ such that:
\[
 i_1=i, \qquad i_p=j, \qquad (A_\delta)_{i_k, i_{k+1}}=1, \text{ for } k=1,\ldots, p-1.
\]
In terms of the elements $a_{km}$ of the matrix $A$ this condition is equivalent to  $a_{i_ki_{k+1}}\neq 0$, for all $k\in\{1, \ldots, p-1\}$ and to $a_{i_1i_2}a_{i_2i_3}\ldots a_{i_{p-1} i_{p}} \neq 0 $. The graph $\mathcal{G}_A$ is connected if there is a path between any pair of vertices $i,j \in \{1, \ldots, n\}$, otherwise it is disconnected. We say that the matrix $A$ is connected (resp., disconnected) whenever $\mathcal{G}_A$ is connected (resp., disconnected) .
The graph $\mathcal{G}_A$ has a cycle (or the matrix $A$ has a cycle) if there is a vertex $i \in \{1, \ldots, n\}$ and a path connecting $i$ to itself. Loops are seen as cycles.  The graph $\mathcal{G}_A$ is cycle-free (or the matrix $A$ is cycle-free) if there are no cycles in  $\mathcal{G}_A$.

If the graph  $\mathcal{G}_A$ is connected and cycle-free then it is a tree, and in this case the number of edges  is $n-1$ \cite[corollary 1.5.3]{Di}. It follows that there exists one vertex with only one edge and, if we eliminate this vertex and its edge, we get a graph with $n-1$ vertices and $n-2$ edges, which is also a tree. If $A$ is a nilpotent matrix and the graph  $\mathcal{G}_A$ is a tree, then there exists a permutation matrix $P$ such that $P^{\top}AP$ is upper triangular. More generally, two matrices $A, A'\in \M_n(\Hq)$ are unitarily equivalent if there exists a unitary $U\in \M_n(\Hq)$ such that $A'=U^*AU$, in which case we write $A'\sim A$. The relation $\sim$ is an equivalence relation. By Schur's triangularization theorem  \cite[theorem 5.3.6]{R}, every matrix $A\in \M_n(\Hq)$ is unitarily equivalent to an upper triangular matrix whose diagonal is complex. By  \cite[theorem 3.5.4]{R}, the numerical range is invariant for the equivalence classes $[A]_{\sim}$, with $A\in \M_n(\Hq)$. Therefore, it is enough to consider upper triangular matrices $A\in \M_n(\Hq)$ with complex diagonal entries.

For the rest of this article we will assume that $A\in \M_n(\Hq)$ is upper triangular.

\section{Circularity of the numerical range}

Our first result shows that the numerical range of a nilpotent matrix is either circular with center at the origin or it is not a disk. In other words, when circular, the disk must be centered at the origin.

Given $A \in \M_n(\Hq)$ there exists an associated complex matrix
\[
\chi(A)=\left[
\begin{array}{cc}
A_1 &  A_2 \\
%\hline
-\bar A_2 & \bar A_1
\end{array}
\right]\in\M_{2n}(\C),
\]
where $A_1, A_2 \in \M_n(\C)$ and $A=A_1+A_2j$.

\begin{thm}\label{center at origin}
Let $A\in\M_n(\bH)$ be nilpotent. If $W(A)$ is a disk, then its center is at the origin.
\end{thm}
\begin{proof}
Since $\chi(AB)=\chi(A)\chi(B)$,\cite[theorem 4.2]{Zh}, $A$ is nilpotent if and only if $\chi(A)$ is nilpotent. When $W(A)$ is circular, it is convex, and according to \cite[theorem 2]{Ye1} and \cite[p. 280]{Ye2},  $W(A) \cap \C= W_{\C}\big(\chi(A)\big)$. Thus $W_{\C}\big(\chi(A)\big)$ is a disk in $\C$. According to \cite[proposition 1]{MM} a nilpotent complex matrix whose numerical range is a disk must have center at the origin. Thus $W(A) \cap \C=W_{\C}(\chi(A))= \bD_{\C}(0, r)$, where $r$ is the radius. We conclude, rebuilding the numerical range by taking the equivalence classes of the elements of the Bild, that $W(A)$ is a disk centered at the origin, i.e. $W(A) = \bD_{\Hq}(0, r)$.
\end{proof}

In theorem \ref{tree circular disk} we will prove that if the graph associated with the nilpotent matrix $A \in \M_n(\Hq)$ has no cycles then the numerical range of $A$ is a disk. When the graph $\mathcal{G}_A$ is disconnected, we can partition the set of vertices into connected components, where each component has no edge to the other components. Then, in terms of the original matrix $A$, we can (through a reordering of the vertices if necessary, \emph{i.e.} through $P^{\top}AP$ where $P$ is a permutation of $\{1,\ldots, n\}$) write $A$ as a block matrix. Now, each block of matrix $A$ is connected. In addition, if $A$ is cycle-free, then each block is cycle-free.

The fact that the quaternionic numerical range is not always convex motivates the following definition.

\begin{defn}
Let $\mathcal{A}_1, \dots, \mathcal{A}_n$ be subsets of $\Hq$. The inter-convex hull of the $\mathcal{A}_i$'s is the set
\[
\iConv \{\mathcal{A}_1, \dots, \mathcal{A}_n\}= \Big\{\sum_i \alpha_i^2 a_i: \, \vc{\alpha}\in \bS_{\R^n} ,\, a_i\in \mathcal{A}_i, \, i=1, \dots, n\Big\}.
\]
\end{defn}

We can easily prove that:

\begin{prop}\label{prop_iconv}
Let $A_1\oplus\ldots \oplus A_k\in \M_n(\Hq)$. Then
\[
W(A_1\oplus\ldots \oplus A_k)=\iConv \Big\{W(A_1),\ldots, W(A_k) \Big\}.
\]
In particular, if $W(A_i)$ is convex, for $i=1, \ldots, k$, then
\[
W(A_1\oplus\ldots \oplus A_k)=\Conv \Big\{W(A_1),\ldots, W(A_k) \Big\},
\]
where $\Conv$ denotes the convex hull.
\end{prop}

Then, to figure out the numerical range of a nilpotent matrix $A$ without cycles, we just need to consider the numerical range of each block $A_i$,  a nilpotent matrix without cycles and connected (that is, a tree). Thus, to establish the relation between the circularity of the numerical range of $A$ and the existence of cycles, we will focus only on connected matrices. We start by an auxiliary and technical result.
%$$\bS^+_{\R^n}=\{(\beta_1,...,\beta_n)\in\bS_{\R^n}:\beta_i\geq 0, 1\leq i\leq n\}.$$
\begin{lemma}\label{lemma_induction}
Let $A=[a_{ij}]_{i,j=1}^n \in \M_n(\bH)$ be a nilpotent and tree matrix and $\vc{\beta}=(\beta_1, \dots, \beta_n)\in \bS^+_{\R^n}$. Then
\[
\bigcup_{\substack{z_k \in \bS_{\bH}\\1\leq k\leq n}}\Bigg\{\sum_{i,j=1}^{n}  \beta_{i}\beta_j z^*_{i}a_{ij} z_j\Bigg\}= \sum_{i,j}^{n} \bS_{\bH}\Big(0, \beta_{i}\beta_j|a_{ij}|\Big).
\]
\end{lemma}

\begin{proof}
 The proof is done by induction in $n$. If $n=1$, then $A$ is the zero matrix and the result is obvious.

Now assume that
\[
\bigcup_{\substack{z_k \in \bS_{\bH}\\1\leq k\leq n-1}}\Bigg\{\sum_{i,j=1}^{n-1} \beta_{i}\beta_j z^*_{i}a_{ij} z_j\Bigg\}= \sum_{i,j}^{n-1} \bS_{\bH}\Big(0, \beta_{i}\beta_j|a_{ij}|\Big).
\]
In the tree $\mathcal{G}_{A}$ associated with $A$, pick any vertex that has only one edge. If necessary change its label to $n$. In this case, since $\mathcal{G}_A$ is a tree, we have that $a_{ni}=0$, for any $i \in \{1, \ldots, n\}$ and, for some $p \in \{1, \ldots, n-1\}$, $a_{pn}\neq 0$, $a_{in}=0$ if $i\in \{1, \ldots, n-1\}\backslash \{p\}$. We then have
\begin{align*}
\bigcup_{\substack{z_k \in \bS_{\bH}\\1\leq k\leq n}} \Big\{\sum_{i,j=1}^{n} \beta_{i}\beta_j z^*_{i}a_{ij} z_j\Big\}= & \bigcup_{\substack{z_k \in \bS_{\bH}\\1\leq k\leq n}} \bigg\{ \sum_{i,j=1}^{n-1} \beta_{i}\beta_j z^*_{i}a_{ij} z_j + \sum_{\max\{i,j\}=n} \beta_{i}\beta_j z^*_{i}a_{ij} z_j\bigg\}\\
=& \bigcup_{\substack{z_k \in \bS_{\bH}\\1\leq k\leq n-1}} \bigcup_{z_n \in \bS_{\bH}} \bigg\{ \sum_{i,j=1}^{n-1} \beta_{i}\beta_j z^*_{i}a_{ij} z_j + \beta_{p}\beta_n z^*_{p}a_{pn} z_n\bigg\}\\
=&  \bigcup_{\substack{z_k \in \bS_{\bH}\\1\leq k\leq n-1}} \bigg\{ \sum_{i,j=1}^{n-1}  \beta_{i}\beta_j z^*_{i}a_{ij} z_j +  \bS_{\bH}\Big(0, \beta_{p}\beta_n|a_{pn}|\Big)\bigg\}\\
=&  \bigcup_{\substack{z_k \in \bS_{\bH}\\1\leq k\leq n-1}} \bigg\{ \sum_{i,j=1}^{n-1}  \beta_{i}\beta_j z^*_{i}a_{ij} z_j\bigg\} +  \bS_{\bH}\Big(0, \beta_{p}\beta_n|a_{pn}|\Big)\\
=&  \bigcup_{\substack{z_k \in \bS_{\bH}\\1\leq k\leq n-1}} \bigg\{ \sum_{i,j=1}^{n-1}  \beta_{i}\beta_j z^*_{i}a_{ij} z_j \bigg\} + \sum_{i=1}^{n-1}  \bS_{\bH}\Big(0, \beta_{i}\beta_n|a_{in}|\Big)\\
\end{align*}
In the second and the last equality we used that the only non-zero $a_{in}$, for $i \in \{1, \ldots, n\}$, is $a_{pn}$ and that all $a_{ni}$ are zero.

By the induction hypothesis, the last equality can be written as
\begin{align*}
\bigcup_{\substack{z_k \in \bS_{\bH}\\1\leq k\leq n}} \Big\{\sum_{i,j=1}^{n} \beta_{i}\beta_j z^*_{i}a_{ij} z_j\Big\}=  & \sum_{i,j}^{n-1} \bS_{\bH}\Big(0, \beta_{i}\beta_j|a_{ij}|\Big) + \sum_{i=1}^{n-1}  \bS_{\bH}\Big(0, \beta_{i}\beta_n|a_{in}|\Big) \\
  = & \sum_{i,j}^{n} \bS_{\bH}\Big(0, \beta_{i}\beta_j|a_{ij}|\Big),
\end{align*}
again using that $a_{nj}=0$ for $j\in \{1, \ldots, n \}$.

Notice that the previous calculation was carried out under the assumption that the $(n-1)\times(n-1)$ matrix had no cycles. In fact, since the initial $n\times n$ matrix $A$ is a tree, as we mentioned before, if we eliminate a one edge vertex together with its edge, we end up with a new graph that is also a tree.
\end{proof}

\begin{thm}\label{tree circular disk}
Let $A\in \M_n(\bH)$ be a nilpotent and tree matrix. Then, $W(A)$ is a circular disk with center at the origin and radius $\max_{ \vc{\beta} \in \bS^+_{\R^n}} \sum_{i,j}^{n} \beta_{i}\beta_j|a_{ij}|$.
\end{thm}
\begin{proof}
 We have:
\[
W(A)=\bigcup_{\vc{x}\in\bS_{\bH^n}}\,\Bigg\{\sum_{i,j=1}^nx^*_ia_{ij}x_j\Bigg\}.
\]
Each summand $x^*_ia_{ij}x_j$ may be written as
\[
\beta_i\beta_jz_i^*a_{ij}z_j
\]
with $\beta_i,\beta_j\geq 0$ and $z_i,z_j\in\bS_{\bH}$. It follows that
$$W(A)=\bigcup_{\vc{\beta}\in\bS^+_{\R^n}}\bigcup_{\substack{z_k \in \bS_{\bH}\\1\leq k\leq n}}\Bigg\{\sum_{i,j=1}^n \beta_i\beta_jz^*_ia_{ij}z_j\Bigg\}.$$

From lemma \ref{lemma_induction}, we have
\begin{equation}\label{W(A) sum spheres}
 W(A)=\bigcup_{\vc{\beta}\in\bS^+_{\R^n}}\,\Bigg\{\sum_{i,j=1}^n  \,\bS_{\bH}(0,\beta_i\beta_j|a_{ij}|)\Bigg\}.
\end{equation}

Therefore, it remains to prove that
\[
\bigcup_{\vc{\beta}\in\bS^+_{\R^n}}\,\Bigg\{\sum_{i,j=1}^n  \,\bS_{\bH}(0,\beta_i\beta_j|a_{ij}|)\Bigg\}=\D_{\bH}\Bigg(0,\max_{ \vc{\beta} \in \bS^+_{\R^n}} \sum_{i,j}^{n} \beta_{i}\beta_j|a_{ij}|\Bigg).
\]
This is achieved by proving a double inclusion. If $y=\sum_{i,j=1}^{n} \beta_{i}\beta_jr_{ij}$ for a given $\vc{\beta}\in\bS^+_{\R^n}$, where $r_{ij}\in\bS_{\bH}(0,|a_{ij}|)$, then
\[
|y|\leq\sum_{i,j=1}^n\beta_i\beta_j|a_{ij}|\leq\max_{ \vc{\beta} \in \bS^+_{\R^n}} \sum_{i,j}^{n} \beta_{i}\beta_j|a_{ij}|.
\]
For the converse inclusion, first observe that the function $f:\bS^+_{\R^n}\to\R$ defined by
\[
f(\vc{\beta})=\sum_{i,j=1}^{n} \beta_{i}\beta_j|a_{ij}|
\]
is a continuous function on a compact set, so it has a maximum at, say, $\vc{\beta}^*\in\bS^+_{\R^n}$. Since $f$ is continuous, $f(1,0,\ldots,0)=0$ and $\bS^+_{\R^n}$ is connected, for every $r_0\in[0,f(\vc{\beta}^*)]$ there exists $\vc{\beta}_0\in\bS^+_{\R^n}$ such that $f(\vc{\beta}_0)=r_0$, that is, for any path connecting $(1,0,\ldots,0)$ to $\vc{\beta}^*$, the values of $f$ run surjectively over the interval $[0,f(\vc{\beta}^*)]$. Take $y\in\D_{\bH}(0,f(\vc{\beta}^*))$. Then, $|y|\leq f(\vc{\beta}^*)$ and $|y|=f(\vc{\beta})$, for some $\vc{\beta}\in\bS^+_{\R^n}$. If $y=0$, then we can take $\vc{\beta}=(1,0,\ldots,0)$ and the inclusion follows. Now, suppose $y\neq 0$. Let $y_{ij}=\beta_i\beta_j|a_{ij}|\dfrac{y}{|y|}$. It follows that
\[
\sum_{i,j=1}^ny_{ij}=f(\vc{\beta})\frac{y}{|y|}=y
\]
and $y_{ij}\in\bS_{\bH}(0,\beta_i\beta_j|a_{ij}|)$. Therefore,
\[
y\in\sum_{i,j=1}^n \bS_{\bH}(0,\beta_i\beta_j|a_{ij}|).
\]
\end{proof}
This result can be applied to more general matrices as shown in the next example.

\begin{ex}\label{example2}
Let $A=\left[
\begin{array}{ccc}
0&  2j&0 \\
0&  0&0\\
0&  0&1
\end{array}
\right]$ and write $A$ as a direct sum
\[
A=A_1\oplus A_2=\left[ \begin{array}{cc}0&  2j\\0&  0\end{array}\right]\oplus\left[ \begin{array}{c} 1\end{array}\right].
\]
By theorem \ref{tree circular disk} we have
\[ W(A_1)=\D_{\bH}(0, 1) \,\,\,\textrm{and}\,\,\,W(A_2)=\{1\}.\]
From proposition \ref{prop_iconv} it follows that
\[W(A)=\mathrm{conv}\,\{W(A_1),W(A_2)\}=\D_{\bH}(0, 1).\]
\end{ex}

The previous result, theorem \ref{tree circular disk}, can be extended to disconnected matrices.

\begin{cor}\label{cor_nil_cycle-free}
Let $A\in\M_n(\bH)$ be a nilpotent and cycle-free matrix. Then, $W(A)$ is a disk with center at the origin.
\end{cor}
\begin{proof}
There exist a permutation matrix $P$ such that $P^{\top}AP=A_1\oplus\ldots\oplus A_k$, where  each $A_i$ is a tree, square matrix. By proposition \ref{prop_iconv} we have
\begin{align*}
  W(A) & = W(P^{\top}AP)=W(A_1\oplus...\oplus A_k) \\
  & =\mathrm{iconv}\,\{W(A_1),...,W(A_k)\\
   & =\mathrm{conv}\,\{W(A_1),...,W(A_k)\}.
\end{align*}
 The result follows from the convexity of $W(A_i)$, see theorem \ref{tree circular disk}.
\end{proof}

We will now give an example that shows the implication of the previous result cannot be strengthened to an equivalence. We will provide a nilpotent (real) matrix $A$ with cycles that have a circular numerical range. The existence of such matrix is supported on two results. Firstly, in theorem 3.7 of \cite{CDM} it was shown that the quaternionic numerical range of a real matrix $A$ is the equivalence classes of the complex numerical range of the matrix $A$, that is
\begin{equation*}
W(A)=\Big[ W_{\C}(A) \Big], \text{ for } A \in M_n(\R).
\end{equation*}
And secondly, theorem $1$ of \cite{CT} provide necessary and sufficient conditions for a $4\times4$ nilpotent complex matrix to have a circular complex numerical range.
\begin{ex}\label{ex_4x4realmatrix}
Let $A \in M_4(\R)$ be
\[
A=\begin{bmatrix}
    0 & 1 & -1 & 0 \\
    0 & 0 & 1 & 1 \\
    0 & 0 & 0& 1\\
    0 & 0 & 0 & 0
\end{bmatrix}.
\]
This matrix satisfies both conditions of \cite{CT} for the numerical range to be circular, therefore the complex numerical range of $A$ is $W_{\C}(A)=\D_{\C}(0, r)$. Since $A$ is real, by \cite{CDM}, $W(A)=\Big[ \D_{\C}(0, r)\Big]=\D_{\Hq}(0, r)$.
\end{ex}

\section{A class of matrices with convex numerical range}
So far we have dealt with nilpotent matrices. In this section we will enlarge our domain and consider matrices that have real entries in the diagonal. It is know that when $A=\alpha I + N\in \M_{n}(\bH)$, for $\alpha \in \R$, the  numerical range is $W(\alpha I+N)=\alpha+W(N)$, see \cite[proposition 3.5.4]{R}. Therefore, if the numerical range of $N$ is convex, so is the numerical range of $A$. We expand this result on convexity for the case where $A$ can be written as the sum of real diagonal matrix $D$ and a nilpotent cycle-free matrix $N$, i.e. $T=N+D$.

Let $D=\diag (d_1, \ldots, d_n) \in \M_n(\R)$ and define
\[
\underline{d}=\min \{d_i: 1\leq i\leq n\}\,\,\,\textrm{ and }\,\,\, \overline{d}=\max \{d_i, 1\leq i\leq n\}.
\]
We will find out in  theorem \ref{W_is_smaller} that the numerical range of $A$ can be decomposed into a union of disks $\D_{\bH}(d, r(d))$, one for each $d \in [\underline{d},\overline{d}]$, that is,
\[
 W(A)= \bigcup_{d \in \big[\underline{d}, \overline{d}\big]}\D_{\bH}(d, r(d)).
\]
To prove the above decomposition, we will need to show that the radius of each disk $r(d)$ varies continuously with the center $d$.

\begin{lemma} \label{disk_continuity}
Let $A\in \mathcal{M}_{n}(\bH)$ be a matrix with real entries in the diagonal. Let
\begin{align*}
f: &\; \R^{n,+} \longrightarrow \R^+ &g :\R^{n,+}  \longrightarrow \big[\underline{d}, \overline{d}\big]  \,\,\,\,\,\,\\
&\,\,\,\,\,\vc{\beta} \mapsto \sum_{i \neq j} \beta_i\beta_j |a_{ij}|  & \vc{\beta} \mapsto \sum_{i} d_{i}\beta_i^2.
\end{align*}
Then, the function $r: \big[\underline{d}, \overline{d}\big] \to \R^+$ defined by
\begin{equation}\label{raio_d}
r(d)= \max \{f(\vc{\beta}): g(\vc{\beta})=d \text{ and }   \vc{\beta} \in \bS^+_{\R^n}\}
\end{equation}
is continuous.
\end{lemma}
\begin{proof}
Define a correspondence $\Gamma:\big[\underline{d}, \overline{d}\big] \rightrightarrows \bS^+_{\R^n}$ to be the intersection of fibers
\[
\Gamma(d)= g^{-1}(d) \cap h^{-1}(0),
\]
where $h:\R^{n,+}\rightarrow \R$ is $h(\vc{\beta})= \|\vc{\beta}\|-1$. We may rewrite function $r$ using the correspondence $\Gamma$ as follows:
\[
r(d)= \max \{f(\vc{\beta}): \vc{\beta} \in \Gamma(d)\}.
\]
According to Berge's maximum theorem, see \cite[p.116]{Be},  $r$ is continuous provided that $f$ and $\Gamma$ are continuous and, for each $d \in \big[\underline{d}, \overline{d}\big]$, $\Gamma(d) \neq \emptyset$. Clearly, $f$ is continuous, and since for each $d \in \big[\underline{d}, \overline{d}\big]$, there is a convex linear combination of $\underline{d}$ and $ \overline{d}$ equal to $d$, $\Gamma(d)$ is nonempty.

We will now prove that $\Gamma$ is continuous by showing that it is sequentially upper and lower semi-continuous. To prove upper semi-continuity, take any convergent sequence $\{\delta_k\}_k$ such that $\delta_k \to d$. We now prove that every sequence  $ \vc{\beta}_k \in \Gamma(\delta_k)$ has a convergent subsequence $\big\{\vc{\beta}_{k_p}\big\}_p$ where  $\vc{\beta}_{k_p}\rightarrow \vc{\beta} \in \Gamma(d)$. In fact, since $\vc{\beta}_k $ is a sequence in the compact set $ \bS^+_{\R^n}$, it has a subsequence that converges to some $\vc{\beta}$. And since $h$ and $g$ are continuous it follows that $\delta_{k_p}=g(\vc{\beta}_{k_p}) \rightarrow g(\vc{\beta})=d$ and $0=h(\vc{\beta}_{k_p}) \rightarrow h(\vc{\beta})$. Thus $\vc{\beta} \in g^{-1}(d) \cap h^{-1}(0)= \Gamma(d)$.

A correspondence is lower semi-continuous if, for any convergent sequence $\{\delta_k\}_k \subseteq \big[\underline{d}, \overline{d}\big]$, such that $\delta_k \to d$, and any $\vc{\beta} \in \Gamma(d)$, there is a convergent sequence $\{\vc{\beta}_k\}_k$,  such that $\vc{\beta}_k \in \Gamma(\delta_k)$ and $\vc{\beta}_k \rightarrow \vc{\beta}$. If  $D=\alpha I$, for some $\alpha \in\R$, then $\underline{d}=\overline{d}$ and there is nothing to prove since the correspondence's domain is a singleton and the correspondence is trivially lower semi-continuous.
When $\underline{d}<\overline{d}$ we need to find for each $\delta_k$ a vector $\vc{\beta}_k$ satisfying $\vc{\beta}_k \in \Gamma(\delta_k)$, and, the whole sequence $\big\{\vc{\beta}_k\big\}_k$, must be such that $\vc{\beta}_k \rightarrow \vc{\beta}$. When $\delta_k=d$ we choose $\vc{\beta}_k=\vc{\beta}$. When $d<\delta_k$, to find $\vc{\beta}_k$ we proceed in the following manner. The vector $\vc{\beta}$ is such that $g(\vc{\beta})=d$, and with it define the sets:%
\begin{align*}
&P(\vc{\beta})=\big\{j \in \{1,\ldots,n\}: \beta_j^2 > 0\big\},\\
&D(\vc{\beta})=\big\{j \in \{1,\ldots,n\}: d_j >d\big\},\\
&d(\vc{\beta})=\big\{j \in \{1,\ldots,n\}: d_j \leq d \big\}.
\end{align*}
Since $d<\delta_k\leq \overline{d}$, $D(\vc{\beta}) \neq \emptyset$. On the other hand, we also have that $P(\vc{\beta}) \cap d(\vc{\beta}) \neq \emptyset$, because $g(\vc{\beta})=d$ is a weighted average of the $d_i$ over the indices  $i \in P(\vc{\beta})$, thus we cannot have them all with strictly higher value than $d$.  We will choose one element from each of these sets, without loss of generality, the element $1$ from  $D(\vc{\beta})$ and the element $2$ from $P(\vc{\beta}) \cap d(\vc{\beta})$. Clearly $d_1>d_2$. Let $0<r^2=\beta_1^2+\beta_2^2\leq 1$ and let the function $\tilde g:[0, 2\pi] \to \R$ be
\[
\tilde g(\theta)= \sum_{i \geq 3} d_i\beta_i^2+ r^2\sin^2(\theta)d_1 + r^2 \cos^2(\theta)d_2.
\]
When $\theta_0=\arcsin \Big(\frac{\beta_1}{\sqrt{\beta_1^2+\beta_2^2}}\Big)$ then $\beta_1^2=r^2\sin^2(\theta_0)$ and $\beta_2^2=r^2\cos^2(\theta_0)$, and  $\tilde g(\theta_0)= g(\vc\beta)$. We have, for $d_1>d_2$,
\begin{align*}
&\tilde g\big(0\big)= \sum_{i \geq 3} d_i\beta_i^2+ r^2d_2 =d+\beta_1^2(d_2-d_1)\equiv \tilde d\leq d, \,\,\,(\beta_1^2\geq 0),\\ %\,\,\,(\beta_1^2\neq 0)\\
&\tilde g\Big(\frac{\pi}{2}\Big)= \sum_{i \geq 3} d_i\beta_i^2+ r^2d_1 =d+\beta_2^2(d_1-d_2)\equiv \hat{ d}>d, \,\,\,(\beta_2^2>0).
\end{align*}
The function $\tilde g$ is continuous and increasing in the interval $(0, \frac{\pi}{2})$, since
$
\tilde g'(\theta)=2\sin(\theta) \cos(\theta)r^2(d_1-d_2)>0.
$
Therefore $\tilde g$ is a homeomorphism (in fact a diffeomorphism) between $[0, \frac{\pi}{2}]$ and the interval $[\tilde d,  \hat{ d}]$. In order to find a vector $\vc{\beta}_k$ such that $g(\vc{\beta_k})=\delta_k$, $k$ must be sufficiently large for $\delta_k \in [\tilde d,  \hat{ d}]$. Since, on one hand $\delta_k \to d$ and, in this case, $d < \delta_k$ and, on the other hand, $\tilde d \leq d <\hat{d}$, there is a $K \in \mathbb{N}$ such that $\delta_k \in [\tilde d,  \hat{ d}]$, for any $k >K$. Then, for $k\leq K$, we can take any $\vc{\beta}_k \in \Gamma(\delta_k)$. For $k>K$, we start by finding $\theta_k$ such that $\tilde g(\theta_k)=\delta_k$. Now, let $\beta_{1,k}=r\sin(\theta_k)$, $\beta_{2,k}=r\cos(\theta_k)$ and $\beta_{i,k}=\beta_{i}$, for $i \in \{3, \ldots, n\}$, that is, $\beta_{1,k}=r\sin(\widetilde{g}^{-1}(\delta_k))$, $\beta_{2,k}=r\cos(\widetilde{g}^{-1}(\delta_k))$ and the remaining terms constant. We can easily verify that $g(\vc{\beta_k})=\widetilde{g}(\theta_k)=\delta_k$. When $\{\delta_k\}_k$ converges to $d$, clearly $\{\theta_k\}_k$ converges to $\theta$ since $\tilde g^{-1}$ is continuous. Therefore $\{\vc{\beta}_k\}_k$ converges to $\vc{\beta}$, as trigonometric functions are continuous.

When $\underline{d} \leq \delta_k<d$ we proceed in a similar way, and conclude that if the set $\{\delta_k:\delta_k<d\}$ is infinite then for each element of this subsequence there is a $\vc{\beta_k} \in \Gamma(\delta_k)$ and these $\vc{\beta_k}$ converge to $\vc{\beta}$.

In conclusion, for any convergent sequence $\{\delta_k\}_k$ we can find a sequence $\vc{\beta}_k\in \bS_{\bH}$ such that $g(\vc{\beta}_k)=\delta_k \rightarrow d=g(\vc{\beta})$.
\end{proof}

We are now able to show that $W(A)$ decompose as unions of disks.
\begin{thm} \label{W_is_smaller}
Let $A=D+N \in \mathcal{M}_{n}(\bH)$, with $D$ a diagonal matrix of real entries and $N$ a nilpotent and tree matrix. Then, we have
\[
W(A)= \bigcup_{d \in \big[\underline{d}, \overline{d}\big]} \D_{\bH}(d, r(d)),
\]
where $r(d)$ is given by (\ref{raio_d}).
\end{thm}
\begin{proof}
We begin by proving that $W(A) \subseteq \bigcup_{d \in \big[\underline{d}, \overline{d}\big]} \mathbb{D}_{\bH}(d, r(d))$. We have:
\begin{align}
W(A)=& \bigcup_{\vc{x} \in \bS_{\bH^n}} \vc{x}^*A\vc{x}= \bigcup_{\vc{x} \in \bS_{\bH^n}} \Big\{ \sum_i d_{i} |x_i|^2 + \sum_{i < j} x^*_ia_{ij}x_j  \Big\}
\nonumber\\
= & \bigcup_{\vc{\beta} \in \bS^+_{\R^n}}  \bigcup_{\substack{z_k \in \bS_{\bH}\\1\leq k\leq n}} \Big\{ \sum_i d_{i} \beta_i^2 + \sum_{i < j} \beta_i\beta_j z_i^* a_{ij}z_j \Big\}\nonumber\\
=& \bigcup_{\vc{\beta} \in \bS^+_{\R^n}}  \Big\{ \sum_i d_{i} \beta_i^2  +  \bigcup_{\substack{z_k \in \bS_{\bH}\\1\leq k\leq n}} \sum_{i < j} \beta_i\beta_j z_i^* a_{ij}z_j \Big\} \nonumber\\
=& \bigcup_{\vc{\beta} \in \bS^+_{\R^n}}  \Big\{ \sum_i d_{i} \beta_i^2 +  \sum_{i < j} \beta_i\beta_j \bS_{\bH}(0, |a_{ij}|) \Big\} \label{1stequality}\\
\subseteq & \bigcup_{\vc{\beta} \in \bS^+_{\R^n}}  \Big\{ \sum_i d_{i} \beta_i^2  +  \sum_{i < j} \beta_i\beta_j \D_{\bH}(0, |a_{ij}|) \Big\} \nonumber\\
= & \bigcup_{d \in \big[\underline{d}, \overline{d}\big]}  d +  \bD_{\bH}(0, r(d)) \label{2ndequality}\\
= & \bigcup_{d \in \big[\underline{d}, \overline{d}\big]} \D_{\bH}(d, r(d)).\nonumber
\end{align}
Equality (\ref{1stequality}) follows from lemma \ref{lemma_induction} applied to matrix $N$. Equality (\ref{2ndequality}) follows from  dividing the set $\bS^+_{\R^n}$ into the fibers of the function $g(\vc{\beta}) =\sum_i d_i \beta_i^2 $, i.e.
\[
\bS^+_{\R^n}=\bigcup_d \Gamma(d).
\]
For each  fiber $g^{-1}(d) \cap \bS^+_{\R^n}=g^{-1}(d) \cap h^{-1}(0)=\Gamma (d)$, with $\underline{d}\leq d \leq \overline{d}$, $h$ and $\Gamma(d)$ defined as in lemma \ref{disk_continuity}, we have
\begin{align*}
\bigcup_{\vc{\beta} \in \Gamma(d) } \Big\{ \sum_i d_{i} \beta_i^2 +  \sum_{i < j} \D_{\bH}(0, \beta_i\beta_j|a_{ij}|) \Big\}  & = \bigcup_{\vc{\beta} \in \Gamma(d) } \bigg\{d +  \sum_{i < j} \D_{\bH}(0, \beta_i\beta_j|a_{ij}|)\bigg\}\\
& = d +  \bigcup_{\vc{\beta} \in \Gamma(d) } \sum_{i < j} \D_{\bH}(0, \beta_i\beta_j|a_{ij}|)\\
 & = d +  \bigcup_{\vc{\beta} \in \Gamma(d) } \D_{\bH}\Big(0, \sum_{i < j} \beta_i\beta_j|a_{ij}|\Big)\\
&=d + \D_{\bH}\Bigg(0, \max_{\vc{\beta}\in\Gamma(d)}   \sum_{i < j}\beta_i\beta_j|a_{ij}|\Bigg) \\
 & = \D_{\bH}(d,r(d)).
\end{align*}

Now, to prove the converse inclusion we need to consider three different cases.

\smallskip

\textbf{Case 1:} $y\in\bS_{\bH}(d, r(d))$. We defined in (\ref{raio_d})
\[
r(d)=\max\{f(\vc{\beta}): \sum_i\beta_i^2d_i=d, \vc{\beta}\in\bS^+_{\R^n}\}
\]
and, in (\ref{1stequality}), we saw that
\[
W(A)=\bigcup_{\vc{\beta} \in \bS^+_{\R^n}}\Big\{ \sum_i \beta_i^2 d_i  +  \sum_{i<j} \beta_i\beta_j \bS_{\bH}(0, |a_{ij}|) \Big\}.
\]
Choose $\vc{\beta}^* \in g^{-1}(d)\cap \bS^+_{\R^n}$ such that $r(d)=f(\vc{\beta}^*)$. To conclude that $y \in W(A)$ we will find $y_{ij} \in  \beta^*_i\beta^*_j \bS_{\bH}(0, |a_{ij}|)$ such that $y-d=\sum_{i<j} y_{ij}$. This  is precisely what we did at the end of theorem \ref{tree circular disk}, taking this time $y_{ij}=\beta^*_i\beta^*_j |a_{ij}| \dfrac{y-d}{|y-d|}$. And we just follow the same reasoning we did there, noting that $f(\vc{\beta}^*)=r(d)=|y-d|$.

\smallskip

\textbf{Case 2.1:} $y\in\D_{\bH}(d, r(d))\backslash \bS_{\bH}(d, r(d))$, with $d=\underline{d}$ or $d=\overline{d}$.

Suppose $y \in \D_{\bH}(\underline{d}, r(\underline{d}))$ and $|y-\underline{d}|<r(\underline{d})$. Notice that for $\vc{\beta}$ be such that $\sum \beta^2_i d_i=\underline{d}$, then $\beta^2_k> 0$ only for those $k$'s for which  $d_k=\underline{d}$. Assume, without loss of generality, that those $k$'s are the first $p$ elements in the diagonal of $A$, that is,
\[\{k \in \{1,\ldots,n\}: d_k=\underline{d}\}=\{1,\ldots,p\}.\]
Define
\[
\mathcal{A}=\{\vc{\beta} \in \bS^+_{\R^n}: \sum \beta^2_i d_i=\underline{d}\}=\bS^+_{\R^p}\times \{0\}^{n-p}.
\]
Then, $\mathcal{A}$ is a connected set and we can take a path form the element where $f$ vanishes (for example some vector $\vc{\beta}$ of the canonical basis) to the element where $f$ is maximum and equal to $r(\underline{d})$, as we did in theorem \ref{tree circular disk}. Since $f$ is continuous, the intermediate value theorem ensures that, in such path, all the values in between $0$ and $r(\underline{d})$ are taken by some element in the path. We can conclude that there exists an element $\vc{\gamma} \in \mathcal{A}$ such that $f(\vc{\gamma})=|y-\underline{d}|$ and that $y \in \bS(\underline{d}, f(\vc{\gamma}))$. By (\ref{1stequality}), $y\in W(A)$. For $d=\overline{d}$ the procedure is analogous.

\smallskip

\textbf{Case 2.2:} $y\in \D_{\bH}(d, r(d))\backslash \bS_{\bH}(d, r(d))$, with $\underline{d}<d<\overline{d}$. Assume that $y \not \in\D_{\bH}(\underline{d}, r(\underline{d}))$, otherwise recur to the previous cases. Then $|y-\underline{d}|>r(\underline{d})$. Let $\rho$ be a function defined over the interval $[\underline{d}, \overline{d}]$ by  $\rho(t)= |y-t|-r(t)$. The function $\rho$ is continuous since the norm is continuous and, by lemma \ref{disk_continuity}, $r$ is continuous. Since $\rho(\underline{d})>0$ and $\rho(d)<0$, continuity of $\rho$ implies the existence of an element $\widetilde{d}$ such that $|y-\widetilde{d}|=r(\widetilde{d})$, and so $y \in  \bS_{\bH}(\widetilde{d}, r(\widetilde{d}))$. This concludes the proof since, by (\ref{1stequality}), $y \in W(A)$.
\end{proof}

The next result identify a class of upper triangular matrices that has convex numerical range.

\begin{cor} \label{prop_diag+shiftlike_convex}
Let $A=D+N\in\M_n(\bH)$, with $D$ a diagonal matrix with real entries, $N$ nilpotent and cycle-free matrix. Then, $W(A)$ is convex.
%\[
%W(A)=\Conv\Big(\bigcup_{d \in \big[\underline{d}, \overline{d}\big]} \D_{\bH}(d, r(d))\Big) ,
%\]
%where $r(d)$ is given by (\ref{raio_d}).
\end{cor}
\begin{proof}
We start assuming that $N$ is a tree. In this case, from theorem \ref{W_is_smaller}, for any $w \in W(A) = \bigcup_{d \in \big[\underline{d}, \overline{d}\big]} \D_{\bH}(d, r(d))$, there is a $d \in \big[\underline{d}, \overline{d}\big] $  such that $w \in \D_{\bH}(d, r(d))$, which implies $|\pi_{\R}(w)-d| \leq r(d)$.  Thus $ d-r(d) \leq \pi_{\R}(w) \leq d+r(d)$ and we end up concluding that
\[
\pi_{\R}(W(A)) \subset W(A) \cap \R= \Big[\min_{d \in [\underline{d}, \overline{d}]} \big( d-r(d)\big), \max_{d \in [\underline{d}, \overline{d}]} \big( d-r(d)\big)\Big]
\]
and since $\pi_{\R}(W(A)) \supset W(A) \cap \R$, we conclude that $\pi_{\R}(W(A)) = W(A) \cap \R$. By \cite[theorem 3]{Ye1} the numerical range is convex.

In the case where $N$ is not a tree, then there exists a permutation matrix $P$ such that $P^{\top}NP=N_1\oplus\ldots\oplus N_k$, where  each $N_i$ is a tree, square matrix. Since $P^{\top}DP$ is still a real diagonal matrix, we have
\begin{align}
  W(A) & = W(D+N)\nonumber \\
  & = W(P^{\top}(D+N)P)=W((D_1+N_1)\oplus...\oplus (D_k+N_k))\nonumber \\
   & =\mathrm{iconv}\,\{W(D_1+N_1),...,W(D_k+N_k)\}\nonumber \\
   & =\mathrm{conv}\,\{W(D_1+N_1),...,W(D_k+N_k)\}.\label{convD+N}
\end{align}
The result follows from proposition \ref{prop_iconv} and the first part of this corollary.
\end{proof}

\begin{ex}
Let $A=\begin{pmatrix} d_1 &  a_{12} & a_{13}\\
                  0 &  d_2 & 0 \\
                   0 &  0 & d_2  \end{pmatrix}$
, where $d_1, d_2\in \R$, $d_1\neq d_2$, and $a_{12}, a_{13}\in \Hq$.
Since $A$ can be written as  $A=d_1 I + (d_2-d_1)\tilde{A}$, with
$
\tilde{A}=\begin{pmatrix}
              0 &  q_{12} & q_{13}\\
                  0 &  1 & 0 \\
                   0 &  0 & 1 \\
            \end{pmatrix},
$
  we have $W(A)=d_1 + (d_2-d_1) W(\tilde{A})$, which means that it is enough to study $W(\tilde{A})$.

By theorem \ref{W_is_smaller}, we have that $W(\tilde{A})=\bigcup_{d \in [0, 1]} \D_{\Hq}(d, r(d))$, with $r(d)=\sqrt{ k d(1-d)}$ and $k={|q_{12}|^2+|q_{13}|^2}$.  Since $W(\tilde{A})\cap \C = \bigcup_{d \in [0, 1]} \D_{\C}(d, r(d))$, we will prove that this union of disks is an ellipse:
\[
\bigcup_{d \in [0, 1]} \D_{\C}(d, \sqrt{k d(1-d)}=\Big\{(x,y)\in \R^2:\frac{(x-\frac{1}{2})^2}{\frac{k+1}{4}}+\frac{y^2}{\frac{k}{4}}\leq 1\Big\}\equiv{\mathcal{E}}.
\]
To show that
${\mathcal{E}} \subseteq \bigcup_{d \in [0, 1]} \D_{\C}(d,\sqrt{k d(1-d)}),$ notice that for $(a,b)\in{\mathcal{E}}$, $(a,b)\in \D_{\C}(d_0,\sqrt{kd_0(1-d_0)})$ with $d_0=\tfrac{2a+k}{2(k+1)} \in [0,1]$.

Conversely, let $(a,b)\in \D_{\C}(d, \sqrt{kd(1-d)})$. We want to show now that $\frac{(a-\frac{1}{2})^2}{\frac{k+1}{4}}+\frac{b^2}{\frac{k}{4}}\leq 1$ , i.e, $\frac{k}{k+1}(a-\frac{1}{2})^2+{b^2}-{\frac{k}{4}}\leq 0$. Since $b^2=kd(1-d)-(a-d)^2$, we have that
\[
\frac{k}{k+1}\Big(a-\frac{1}{2}\Big)^2+{b^2}-{\frac{k}{4}}=-\frac{1}{k+1}((k+1)d-a)^2+\frac{k}{k+1}((k+1)d-a)-\frac{k^2}{4(k+1)}.
\]
This is a second degree polynomial in $(k+1)d-a$, with down concavity and always non-positive, so we conclude that $(a,b)\in \mathcal{E}$.
\end{ex}

\section{Convexity and circularity of $3\times 3$ nilpotent matrices}
Our main goal in this section is to establish necessary and sufficient conditions for  quaternionic $3\times 3$ nilpotent matrices to have circular or convex numerical range. We start by finding out that this condition for circularity is related to the product of all non-zero elements of the matrix. In par\-ti\-cular, it relates the numerical range's circularity with the product $a^*_{13}a_{12}a_{23}$ vanishing or not. Theorem \ref{NSC 3x3 convex} gives a  condition for the convexity of the numerical range in terms of the values assumed by exactly the same product. Consequently, the numerical range is convex if and only if $a^*_{13}a_{12}a_{23} \in \R$.

%%%%%%%%%%%%%%%%%%%%%%
%%%%%%%%%%%%%%%%%%%%%

\begin{thm} \label{NR_3x3_disk}
Let $A\in \M_3(\bH)$ be a nilpotent matrix. Then, $W(A)$ is a disk with center at the origin, if and only if, $A$ is cycle-free.
\end{thm}
\begin{proof}
Sufficiency was proved in corollary \ref{cor_nil_cycle-free}. For necessity we will show that if $A$ has a cycle then $W(A)$ is not a disk. We just need to observe that there are some elements $q \in \bS_\bH$ for which $aq \in W(A)$, with
\[
a=\max \{ \beta_1\beta_2|a_{12}|+\beta_1\beta_3|a_{13}|+\beta_2\beta_3|a_{23}|\},
\]
and some others $\tilde q\in \bS_\bH$ for which $a\widetilde{q}\notin W(A)$. If $w\in W(A)$ then, by the triangle inequality, $|w|\leq a$, and the equality holds if, and only if, all the terms of $\vc{x}^*A\vc{x}$ are collinear. We have that $aq \in W(A)$ if, and only if, any term of $\vc{x}^*A\vc{x}$ is in the real span of $q$, for some $q\in\bS_{\bH}$, that is, writing $a_{ij}=|a_{ij}|w_{ij}$ with $w_{ij}\in\bS_{\bH}$,
\[
\begin{cases}
  \beta_1\beta_2|a_{12}| z_1^* w_{12} z_2 & =\beta_1\beta_2|a_{12}| q \\
  \beta_1\beta_3|a_{13}| z_1^* w_{13} z_3 & =\beta_1\beta_3|a_{13}| q\\
  \beta_2\beta_3|a_{23}| z_2^* w_{23} z_3 & =\beta_2\beta_3|a_{23}| q
\end{cases}
\]
If $a_{13}a_{12}a_{23} \neq 0$ (i.e, $A$ has cycles) the system has solution only  when $q=z_3^* w_{13}^*w_{12}w_{23} z_3$, for some $z_3\in\bS_{\bH}$. That is, only for those $q \in \big[w_{13}^*w_{12}w_{23}\big]$ can we reach the maximum $aq \in W(A) $. But then $W(A) $ is not circular since for $\tilde q\notin \big[w_{13}^*w_{12}w_{23}\big]$, $a \tilde q \notin W(A)$ but $a q \in W(A)$ for all $q \in \big[w_{13}^*w_{12}w_{23}\big]$.
 We conclude that when $A$ has cycles the numerical range is not circular.
\end{proof}

We will now obtain a similar result for the convexity of the numerical range of a matrix $A$ in $\M_3(\Hq)$, that is, $W(A)$ is convex, if and only if, $a^*_{13}a_{12}a_{23}\in\R$. The argument uses two known results: \cite[Theorem 3]{Ye1}, which says that $W(A)$ is convex if, and only if, $W(A)\cap \R =\pi_{\R}\big(W(A)\big)$, and that $W(A)\cap \R$ is a closed interval, see \cite[Corollary 1]{Ye1}. Since the numerical range is connected \cite[Theorem 3.10.7]{R} we know that $\pi_{\R}\big(W(A)\big)=[m,M]$ with
\begin{align*}
m & =  \min \pi_\R (W(A)),\\
M & =  \max \pi_\R (W(A)).
\end{align*}
Thus we can conclude that the numerical range is convex if, and only if,  $W(A)\cap \R=[m,M]$. That is, $W(A)$ is convex if, and only if, there are $\vc{v}$ and $\vc{\hat{v}}$ in $\bS_{\Hq^3}$, such that
\begin{align}\label{v_hat{v}}
\vc{v}^*A\vc{v}=M \,\,\, & \textrm{and}\,\,\,\vc{\hat{v}}^*A\vc{\hat{v}}=m.
\end{align}
Necessarily, there are $\vc{y}, \vc{\hat{y}}\in \bS_{\bH^3}$,  such that $M= \pi_\R (\vc{y}^*A\vc{y})$ and $m= \pi_\R (\vc{\hat{y}}^*A\vc{\hat{y}})$. The following lemma is preparatory to reach the conclusion in (\ref{v_hat{v}}).

\begin{lemma}\label{prep_for_general}
Let $A\in \M_3(\bH)$ be a nilpotent matrix satisfying the following condition:
\begin{align}\label{R-linearly_independent}
\textrm{the two quaternions}\,\,\,a_{12},\, a_{13}a_{23}^*\,\,\, & \textrm{are}\,\,\R-\textrm{linearly independent}.
\end{align}
Suppose that $\pi_\R (\vc{y}^*A\vc{y})=M$. Then,
\[
y^*_1(a_{12}y_2+a_{13}y_3), y^*_2(a_{12}^*y_1+a_{23}y_3), (y^*_1 a_{13}+y^*_2 a_{23})y_3 \in \R\setminus\{0\}
\]
\end{lemma}
\begin{proof}
To prove that $\omega=y^*_1(a_{12}y_2+a_{13}y_3)\in \R$ we start by writing
\begin{align*}\label{maximum projection}
  M=&\pi_\R (y^*_1(a_{12}y_2+a_{13}y_3)+y^*_2a_{23}y_3)\\
   =&\pi_\R (\omega+y^*_2a_{23}y_3).
\end{align*}
%
%If $\omega\notin \R$,
There exists $z\in \bS_{\bH}$ such that
\[
z^* \omega=|\omega|\in\R
\]
%and therefore
%\[
%\pi_\R(z^*\omega)=z^*\omega>\pi_\R(\omega).
%\]
Taking now $\vc{\tilde{y}}=(y_1 z, y_2, y_3)\in \bS_{\bH^3}$, we have:
\begin{align*}
M=\pi_\R(\vc{y}^*A\vc{y}) \geq & \pi_\R(\vc{\tilde{y}}^*A\vc{\tilde{y}}).
\end{align*}
Then,
\begin{align*}
M=\pi_\R(\omega+y^*_2a_{23}y_3) \geq & \pi_\R(z^* \omega+y^*_2a_{23}y_3)
\end{align*}
and by $\R$-linearity of $\pi_\R$ we have
\begin{align*}
\pi_\R(\omega) \geq & \pi_\R(z^* \omega)=|\omega|.
\end{align*}
We conclude that $\omega=|\omega|$ and so $\omega\in\R$.

Now we prove that $y^*_1(a_{12}y_2+a_{13}y_3) \neq 0$ by assuming that $y^*_1(a_{12}y_2+a_{13}y_3) = 0$ and then finding a vector $\vc{t} \in \bS_{\bH^3}$ with $\vc{t}^*A\vc{t}> \pi_{\R}\big(\vc{y}^*A\vc{y}\big)$, reaching a contradiction with $\vc{y}$ being the maximizer of $\pi_{\R}\big(\vc{x}^*A\vc{x}\big)$ for $\vc{x} \in \bS_{\bH^3}$.
Condition (\ref{R-linearly_independent}) implies that $a_{23}\neq 0$ and $a\equiv a_{12}+a_{13}\dfrac{a_{23}^*}{|a_{23}|}\neq 0 $. If $y^*_1(a_{12}y_2+a_{13}y_3) = 0$ then %
\[
\pi_{\R}\big(\vc{y}^*A\vc{y}\big)=y_2^*a_{23}y_3= \dfrac{|a_{23}|}{2}
\]
The previous equality comes from the maximum of $f(\alpha_1,\alpha_2)=\alpha_1\alpha_2$ , subject to $\alpha_1^2+\alpha_2^2=1$, being $\dfrac{1}{2}$. Let $t_1=\beta_1 \in \R$, $t_2=\beta_2 \dfrac{a^*}{|a|}$ and $t_3=\dfrac{a_{23}^*}{|a_{23}|}t_2$, with $\beta_2=\dfrac{\sqrt{2}}{2}-\epsilon$ and $\beta_1^2+2\beta^2_2=1$.
\begin{align*}
\vc{t}^*A\vc{t}-\dfrac{|a_{23}|}{2}&=\beta_1\beta_2|a|+\beta_2^2|a_{23}|-\dfrac{|a_{23}|}{2}\\
&=\sqrt{1-2\beta_2^2}\beta_2|a|+\big(\beta_2^2-\dfrac{1}{2}\big)|a_{23}|\\
&= \sqrt{2\sqrt{2}\epsilon-2\epsilon^2}\big(\dfrac{\sqrt{2}}{2}-\epsilon\big)|a|-\big(\sqrt{2}-\epsilon\big)\epsilon|a_{23}|\\
&=\epsilon \bigg\{ \sqrt{ \dfrac{2\sqrt{2}}{\epsilon}-2}(\dfrac{\sqrt{2}}{2}-\epsilon\big)|a|-\big(\sqrt{2}-\epsilon\big)|a_{23}|\bigg\}
\end{align*}
Clearly $\dfrac{1}{\epsilon} \Big(\vc{t}^*A\vc{t}-\dfrac{|a_{23}|}{2}\Big)>0$ for very small and positive $\epsilon$. Thus, $\vc{t}^*A\vc{t}>\pi_{\R}\big(\vc{y}^*A\vc{y}\big)$, and we found a contradiction.

The second case, $y^*_2(a_{12}^*y_1+a_{23}y_3)\in \R\setminus\{0\}$, follows from writing $M$ as follows
 \begin{align*}\label{maximum projection}
  M=&\pi_\R (y^*_1a_{12}y_2+ y^*_1a_{13}y_3+y^*_2a_{23}y_3)\\
  = &\pi_\R(y^*_1a_{12}y_2)+ \pi_\R (y^*_1a_{13}y_3)+\pi_\R(y^*_2a_{23}y_3)\\
  = &\pi_\R(y^*_2a_{12}^*y_1)+ \pi_\R (y^*_1a_{13}y_3)+\pi_\R(y^*_2a_{23}y_3)\\
  = &\pi_\R\big(y^*_2(a_{12}^*y_1+a_{23}y_3)\big)+ \pi_\R (y^*_1a_{13}y_3).
 \end{align*}
Now we proceed as in the first case to find out that $y^*_2(a_{12}^*y_1+a_{23}y_3)\in \R$. To prove that $y^*_2(a_{12}^*y_1+a_{23}y_3)\neq 0$ we let this time $a\equiv a_{23}+a_{12}^*\dfrac{a_{13}}{|a_{13}|}$. From (\ref{R-linearly_independent}) we have that $a_{13}\neq 0$ and $a\neq 0 $. We will find a contradiction, in the same way as in the first case, assuming that $y^*_2(a_{12}^*y_1+a_{23}y_3)=0$. In this case it must be that $\pi_{\R}\big(\vc{y}^*A\vc{y}\big)= \dfrac{|a_{13}|}{2}$, and we will find a  $\vc{t} \in \bS_{\bH^3}$ with $\vc{t}^*A\vc{t}> \dfrac{|a_{13}|}{2}$. Such $\vc{t}$ has $t_1=\dfrac{a_{13}}{|a_{13}|}t_3$, $t_2=\beta_2$, $t_3=\beta_3 \dfrac{a^*}{|a|}$, with $\beta_3=\dfrac{\sqrt{2}}{2}-\epsilon$ and $\beta_2^2+2\beta^2_3=1$. The rest of the proof proceeds just like the first case.

The proof for case $3$ mimics the previous two.

\end{proof}
Now we will state and prove a necessary and sufficient condition for a nilpotent $3\times 3$ matrix to have convex numerical range.

\begin{thm}\label{NSC 3x3 convex}
Let $A\in \M_3(\bH)$ be a nilpotent matrix. Then, $W(A)$ is convex if, and only if, $a^*_{13}a_{12}a_{23}\in\R$.
\end{thm}
\begin{proof}
First, consider that $a_{13}^*a_{12}a_{23} \in \R$. The case where $a_{13}^*a_{12}a_{23}=0$ was dealt in theorem \ref{NR_3x3_disk}, the numerical range is circular and therefore convex. For the other cases, the matrix $A$ is unitary equivalent to a real matrix, \emph{i.e} there exists an unitary matrix $U \in \mathcal{M}_n(\Hq)$, such that $U^*AU \in \mathcal{M}_n(\R)$. By \cite[theorem 3.6]{CDM}, we know that the numerical range of any real matrix is convex, thus $W(A)=W(U^*AU)$ is convex.
For the unitary matrix $U$, take the diagonal matrix $\diag (\rho, z_{12}^*\rho, z_{13}^*\rho)$, where $\rho \in \bS_{\Hq}$ and $z_{ij} \in \bS_{\Hq}$ are such that $a_{ij}=|a_{ij}|z_{ij}$. Its now a matter of simple calculations, using that $z_{13}^*z_{12}z_{23}=\pm1$, to check that  $U^*AU \in \mathcal{M}_n(\R)$.

Now we consider the converse implication, that is, if $W(A)$ is convex then $a^*_{13}a_{12}a_{23}\in \R$. If $a_{12},\,a_{13}a_{23}^*$ are $\R$-linearly dependent we easily see that $a_{13}a_{23}^*a_{12}^* \in \R$ and, since $\pi_{\R}(ab)=\pi_{\R}(ba)$, then $a_{13}^*a_{12}a_{23} \in \R$. Therefore, we can assume that $a_{12},\,a_{13}a_{23}^*$ are $\R$-linearly independent and, by lemma \ref{prep_for_general}, conclude that
\begin{equation}\label{equation_y2_y_3}
y^*_1(a_{12}y_2+a_{13}y_3), y^*_2(a_{12}^*y_1+a_{23}y_3), (y^*_1 a_{13}+y^*_2 a_{23})y_3\in \R\setminus\{0\} .
\end{equation}
Hence, for some $\alpha_1, \alpha_2, \alpha_3\in \R\setminus\{0\}$ we can write
\begin{equation}\label{system}
 \left\{
 \begin{array}{c}
  y_1= \alpha_1 (a_{12}y_2+a_{13}y_3) \\
  y_2= \alpha_2 (a_{12}^*y_1+a_{23}y_3)\\
  y_3= \alpha_3 (a_{13}^*y_1+a_{23}^*y_2)
 \end{array}
\right.
\end{equation}

Substituting $y_1$ in the second equation, we get
\begin{equation}\label{eq1}
 (1-\alpha_1\alpha_2|a_{12}|^2)y_2=\alpha_2(\alpha_1 a_{12}^*a_{13}+a_{23})y_3.
\end{equation}

Suppose $1-\alpha_1\alpha_2|a_{12}|^2\neq 0$. We have
\[
y_2=r_2(\alpha_1 a_{12}^*a_{13}+a_{23})y_3, \quad \text{where} \quad r_2=\frac{\alpha_2}{1-\alpha_1\alpha_2|a_{12}|^2}.\]

Therefore, for  $\vc{y}=(y_1, y_2, y_3)\in \bS_{\bH^3}$ such that $M=\pi_\R(\vc{y}^*A\vc{y})$,
\begin{align}
  \vc{y}^*A\vc{y} & ={y}^*_1(a_{12}{y}_2+a_{13}{y}_3)+{y}^*_2a_{23}{y}_3 \nonumber \\
 & ={y}^*_1(a_{12}{y}_2+a_{13}{y}_3)+ r_2 |a_{23}|^2 |{y}_3|^2 + r_2\alpha_1 {y}_3^*a_{13}^*a_{12}a_{23}{y}_3. \label{eq2}
\end{align}
Notice that the first two terms of (\ref{eq2}) are real. Since $W(A)$ is convex, by \cite[theorem 3]{Ye1}, $ M= \vc{y}^*A\vc{y} \in W(A)\cap \R$. Thus, the term ${y}_3^*a_{13}^*a_{12}a_{23}{y}_3$ is also real. This only happens if $a_{13}^*a_{12}a_{23} \in \R$ (since ${y}_3^*a_{13}^*a_{12}a_{23}{y}_3=|y_3|^2\frac{{y}_3^*}{|y_3|}a_{13}^*a_{12}a_{23}\frac{{y}_3}{|y_3|}$ and $\frac{{y}_3^*}{|y_3|}a_{13}^*a_{12}a_{23}\frac{{y}_3}{|y_3|}\sim a_{13}^*a_{12}a_{23}$) or $y_3=0$. The case $y_3=0$ can be ruled out because then  $(y^*_1 a_{13}+y^*_2 a_{23})y_3=0$ and this contradicts (\ref{equation_y2_y_3}).

If $1-\alpha_1\alpha_2|a_{12}|^2=0$ and since $y_3\neq 0$, from (\ref{eq1}) $\alpha_1a^*_{12}a_{13}+a_{23}=0$. It follows that $a^*_{13}a_{12}a_{23}\in\R$.
\end{proof}

We finish with a simple example.

\begin{ex}
Let $A=\begin{pmatrix}0 & i & j  \\ 0 & 0 & k \\ 0 & 0 & 0 \end{pmatrix}$. Since $(-j)ik=-1$, $W(A)$ is convex and noncircular.
\end{ex}


\begin{thebibliography}{99}
\bibitem[AY1]{Ye1} Y. Au-Yeung ,\emph{ On the convexity of the numerical range in quaternionic Hilbert space}, Linear and
Multilinear Algebra, \textbf{16} (1984), 93--100.

\bibitem[AY2]{Ye2} Y. Au-Yeung ,\emph{ A short proof of a theorem on the numerical range of a normal
quaternionic matrix}, Linear and Multilinear Algebra, \textbf{39:3} (1995), 279--284.
%DOI: 10.1080/03081089508818402

\bibitem[Be]{Be} C. Berge, \emph{Topological Spaces}, Dover, 1997.

\bibitem[CDM]{CDM} L. Carvalho, C. Diogo, S. Mendes, \emph{A bridge between quaternionic and complex numerical ranges}. (To appear in  Linear Algebra and its Applications).    %arXiv:1904.02757, (2019).

\bibitem[CT]{CT} M-T. Chien, B-S. Tam, \emph{Circularity of the numerical range}, Linear Algebra and its Applications, \textbf{201} (1994), 113--133.

\bibitem[Di]{Di}  R. Diestel, \emph{Graph Theory}, 4th edition, Springer, 2010.

\bibitem[Ed]{Edwards} C. Edwards, \emph{Advanced Calculus of Several Variables}, Dover, 1994.

\bibitem[GR]{GR} K. Gustafson, D. Rao, \emph{Numerical Range}, Springer-Verlag, New York, 1997.

\bibitem[Ki]{Ki} R. Kippenhahn, \emph{On the numerical range of a matrix}, Translated from the German by Paul F. Zachlin and Michiel E. Hochstenbach. Linear Multilinear Algebra, \textbf{56:1-2} (2008), 185-225.

\bibitem[MM]{MM} V. Matache, M. Matache, \emph{When is the numerical range of a nilpotent matrix circular?}, Applied Mathematics and Computation, \textbf{216(1)} (2010), 269--275.

\bibitem[R]{R}   L. Rodman, \emph{Topics in Quaternion Linear Algebra}, Princeton University Press, 2014.

\bibitem[ST]{ST} W. So, R. C. Thompson, \emph{Convexity of the upper complex plane part of the numerical range of a quaternionic matrix}, Linear and Multilinear Algebra, \textbf{41} (1996), 303--365.

\bibitem[STZ]{STZ} W. So, R. C. Thompson, F. Zhang, \emph{The numerical range of normal matrices with quaternion entries}, Linear and Multilinear Algebra, \textbf{37} (1994), 175--195.

\bibitem[Zh]{Zh} F. Zhang, \emph{Quaternions and matrices of quaternions}, Linear Algebra and its Applications, \textbf{251} (1997), 21--57.

%\bibitem[Ki]{Ki} Kippenhahn, R., Uber die wertvorrat einer matrix. Math. Nachr. 1951;6:193–228.










\end{thebibliography}
\end{document}